\documentclass[11pt]{amsart}

\usepackage{amsmath}
\usepackage{amsthm}
\usepackage{amssymb}
\usepackage{lineno,hyperref}
\modulolinenumbers[5]

\newtheorem{thm}{Theorem}[section]

\newtheorem{lem}[thm]{Lemma}
\newtheorem{prop}{Proposition}

\newtheorem{rem}{Remark}
\newcommand{\al}{\alpha}

\newcommand{\ep}{\varepsilon}

\def\wt{\widetilde}
\def\re{\mathbb{R}}

\def\|{\Vert}

\def\({\left(}
\def\){\right)}
\def\e{\varepsilon}

\def\l{\lambda}
\def\e{\varepsilon}
\def\cj{\overline}

\begin{document}


\title[Hardy-Sobolev critical equations with the multiple singularities]{On the Neumann Problem of Hardy-Sobolev critical equations with the multiple singularities}

\author{Masato Hashizume}

\address{Department of Mathematics, Graduate School of Science, Osaka City University\\
		3-3-138 Sugimoto Sumiyoshi-ku, Osaka-shi\\
		Osaka 558-8585 Japan}
\email{d15san0f06@st.osaka-cu.ac.jp}
		
\author{Chun-Hsiung Hsia}
\address{Department of Mathematics, Institute of Applied Mathematical Sciences\\
		National Center for Theoretical Sciences, National Taiwan University,\\
		No. 1, Sec. 4, Roosevelt Rd, Taipei 10617, Taiwan}
\email{willhsia@math.ntu.edu.tw}

\author{Gyeongha Hwang}
\address{National Center for Theoretical Sciences\\
		No. 1 Sec. 4 Roosevelt Rd., National Taiwan University\\
		Taipei, 10617, Taiwan}
\email{ghhwang@ncts.ntu.edu.tw}
 
\begin{abstract}
Let $N \geq 3$ and $\Omega \subset \mathbb{R}^N$ be $C^2$ bounded domain. We study the existence of positive solution $u \in H^1(\Omega)$ of 
\begin{align*}
\left\{
\begin{array}{l}
-\Delta u + \lambda u = \frac{|u|^{2^*(s)-2}u}{|x-x_1|^s} + \frac{|u|^{2^*(s)-2}u}{|x-x_2|^s}\text{ in }\Omega\\
\frac{\partial u}{\partial \nu} = 0 \text{ on }\partial\Omega,
\end{array}\right.
\end{align*}
where $0 < s <2$, $2^*(s) = \frac{2(N-s)}{N-2}$ and $x_1, x_2 \in \overline{\Omega}$ with $x_1 \neq x_2$. First, we show the existence of positive solutions to the equation provided the positive $\lambda$ is small enough. In case that one of the singularities locates on the boundary and the mean curvature of the boundary at this singularity is positive, the existence of positive solutions is always obtained for any $\lambda > 0$. Furthermore, we extend the existence theory of solutions to the equations for the case of the multiple singularities with different exponents.
\end{abstract}




\maketitle

\section{Introduction}
The Hardy-Sobolev inequality asserts that for all $u \in H^1_0(\mathbb{R}^N)$, there exists a positive constant $C=C(N,s)$ such that
\begin{align}\label{main ineq}
C\left(\int_{\mathbb{R}^N}\frac{|u|^{2^*(s)}}{|x|^s}dx\right)^{\frac{2}{2^*(s)}} \leq \int_{\mathbb{R}^N}|\nabla u|^2 dx
\end{align}
where $N \geq 3$, $0 < s <2$ and $2^*(s) = \frac{2(N-s)}{N-2}$. Suppose $\Omega \subset \mathbb{R}^N$, then the Hardy-Sobolev inequality holds for $u \in H^1_0(\Omega)$. The best constant of the Hardy-Sobolev inequality is defined as
\[S_s(\Omega) := \inf_{u \in H^1_0(\Omega) \setminus \{0\}} \frac{\int_{\Omega} |\nabla u|^2 dx}{\left(\int_{\Omega} \frac{|u|^{2^*(s)}}{|x|^s} dx\right)^{\frac{2}{2^*(s)}}}.\]
It is easy to see, up to a scaling, that the minimizer for $S_s(\Omega)$ is a least-energy solution of the Euler-Lagrangian equation:
\begin{align}\label{EL eq}
\left\{
\begin{array}{l l}
-\Delta u = \frac{|u|^{2^*(s)-2}u}{|x|^s},\ u>0 &\text{ in } \Omega\\
u = 0 &\text{ on } \partial \Omega.
\end{array}
\right.
\end{align}

When $\Omega = \mathbb{R}^N$, $S_s(\mathbb{R}^N)$ is attained by
\[g_a(x) = \left(a(N-s)(N-2)\right)^{\frac{N-2}{2(N-s)}}(a+|x|)^{\frac{2-N}{2-s}},\]
for some $a>0$ (see \cite{ghyu, lieb}). Moreover, $g_a(x)$ are the only positive solutions to \eqref{EL eq}. Hence, in case $0 \in \Omega$, by a standard scaling invariance argument, it is easy to see $S_s(\Omega) = S_s(\mathbb{R}^N)$ and $S_s(\Omega)$ cannot be attained unless $\Omega = \mathbb{R}^N$. However, if $0 \in \partial \Omega$, the existence of the minimizer for $S_s(\Omega)$ is established under the assumption that the mean curvature of $\partial \Omega$ at $0$, $H(0)$ is negative (see \cite{ghro}). 

Concerning the Dirichlet problem, the second author and his collabolators \cite{hlw} showed the existence of solutions to the equation 
\[-\Delta u = \lambda u^{\frac{N+2}{N-2}} + \frac{|u|^{2^*(s)-2}u}{|x|^s}, u>0\text{ in } \Omega\]
for $\lambda > 0$. Furthermore, Li-Lin \cite{lilin} proved the existence of the least energy solution to the equation involving two Hardy-Sobolev critical exponents
\[-\Delta u = \lambda\frac{u^{2^*(s_1)-1}}{|x|^{s_1}} + \frac{u^{2^*(s_2)-1}}{|x|^{s_2}},\ u > 0 \text{ in }\Omega\]
where $0 < s_2 < s_1 < 2$ and $0 \neq \lambda \in \mathbb{R}$. For intersted readers, see also \cite{bpz, czz, mus}.

Regarding the Neumann problem 
\begin{align}\label{NP eq}
\left\{
\begin{array}{l}
-\Delta u + \lambda u = \frac{|u|^{2^*(s)-2}u}{|x|^s}, u > 0\text{ in }\Omega,\\
\frac{\partial u}{\partial \nu} = 0 \text{ on }\partial\Omega,
\end{array}\right.
\end{align}
we first notice that if $\lambda \leq 0$, then integration of \eqref{NP eq} over $\Omega$ gives
\[ 0 < \int_\Omega  \frac{|u|^{2^*(s)-2}u}{|x|^s} dx = \int_{\Omega} -\Delta u + \lambda u dx \leq 0.\]
Hence, there does not exist a positive solution to \eqref{NP eq}. So, only the case where $\lambda > 0$ are adderessed in literature. In this case, Ghossoub-Kang \cite{ghka} showed that \eqref{NP eq} has a positive solution if the mean curvature of $\partial \Omega$ at $0$, $H(0) $ is positive. Furthermore, Chabrowski \cite{cha} investigated the solvability of the nonlinear Neumann problem with indefinite weight functions
\[-\Delta u + \lambda u = \frac{Q(x)|u|^{2^*(s)-2}u}{|x|^s}, u > 0\text{ in }\Omega\]
and gives some sufficient condition on $Q(x)$ provided the mean curvature of $\partial \Omega$ at $0$, $H(0) >0$. 
Recently, concerning the equation (\ref{NP eq}) the first author investigated the case when $H(0) \leq 0$ in \cite{Hashizume}. He showed the existence of $\l_*$ such that for $\l \in (0,\l_*)$, a least energy solution of (\ref{NP eq}) exists, and when $\l > \l_*$ a least energy solution does not exist.
We remark that the sufficient conditions for Dirichlet and Neumann problems are completely different.

In this paper, we consider the Neumann problem with the multiple singularities
\begin{align}\label{main in}
\left\{
\begin{array}{l}
-\Delta u + \lambda u = \frac{|u|^{2^*(s)-2}u}{|x-x_1|^s} + \frac{|u|^{2^*(s)-2}u}{|x-x_2|^s}\text{ in }\Omega\\
\frac{\partial u}{\partial \nu} = 0 \text{ on }\partial\Omega
\end{array}\right.
\end{align}
where $\Omega$ is $C^2$-bounded domain and $x_1, x_2 \in \overline{\Omega}$ with $x_1 \neq x_2$.

The main results of this article are as follows
\begin{thm}[Existence of solution to \eqref{main in} for small $\lambda$]\label{thm1}
There exists $\Lambda > 0$ such that the equation \eqref{main in} has a positive solution provided the positive parameter $\lambda < \Lambda$.
\end{thm}

\begin{thm}[Existence of solution to \eqref{main in} with the boundary singularity]\label{thm2}
Suppose $x_1 \in \partial \Omega$ and the mean curvature of $\partial \Omega$ at $x_1$, $H(x_1)$ is positive. Then there exists a positive solution to $\eqref{main in}$.
\end{thm}
Note that Theorem \ref{thm2} asserts the singularity at boundary prevails the singularity in the interior.
To establish the existence theory, we study the functional 
\begin{align}\label{energy}
J_\lambda(u) =  \int_{\Omega}\frac12(|\nabla u|^2 + \lambda u^2) - \frac{1}{2^*(s)}\left(\frac{u_+^{2^*(s)}}{|x-x_1|^s}+\frac{u_+^{2^*(s)}}{|x-x_2|^s}\right)dx
\end{align}
defined on $H^1(\Omega)$ where $u_+ = \max(u,0)$. It is not hard to see that $J_\lambda$ is a $C^1$ functional and
\begin{align}\label{der_energy}
\langle J'_\lambda(u), \phi \rangle = \int_\Omega \nabla u \nabla \phi + \lambda u \phi - \left(\frac{u_+^{2^*(s)-1}}{|x-x_1|^s}\phi+\frac{u_+^{2^*(s)-1}}{|x-x_2|^s}\phi\right)dx
\end{align}
for $\phi \in H^1(\Omega)$. Moreover, by Sobolev embedding theorem, we obtain
\begin{align*}
J_\lambda(u) &= \int_{\Omega}\frac12(|\nabla u|^2 + \lambda u^2) - \frac{1}{2^*(s)}\left(\frac{u_+^{2^*(s)}}{|x-x_1|^s}+\frac{u_+^{2^*(s)}}{|x-x_2|^s}\right)dx\\
&\geq \int_{\Omega}\frac12(|\nabla u|^2 + \lambda u^2)dx - \frac{C_0}{2^*(s)}\left(C_{s,\delta}\int_\Omega |\nabla u|^2dx + \widetilde{c}(\delta)\int_{\Omega} \lambda u^2dx\right)^{\frac{2^*(s)}{2}}.
\end{align*}
Hence there exists $\al > 0$ and $\rho  > 0$ such that
\[J_\lambda(u) \geq \al \text{ if } \|u\|=\rho.\]
The scenario for the proof of the theorems is to apply the mountain pass lemma to attack the existence theory. However, the crux is to decide the threshold of the energy level so that the Palais-Smale condition would hold. We use concentration compactness principle to find this energy level.

\begin{rem}
The existence problem for \eqref{main in} with $x_1, x_2 \in \Omega$ and $\Lambda < \lambda$ is still open.
\end{rem}

In section 2, we investigate the threshold of the Palais-Smale condition for $J_\lambda$. In section 3 and 4, we prove the existence of solutions as described in Theorem \ref{thm1} and Theorem \ref{thm2}, respectively. In section 5, the positivity of solutions is established. In section 6, regularity of solution is considered. Lastly in section 7, we give brief accounts for the Neumann problem with the multiple singularities. Namely, the existence of solutions to
\begin{align*}
\left\{
\begin{array}{l}
-\Delta u + \lambda u = \sum_{i=1}^I\frac{|u|^{2^*(s_i)-2}u}{|x-x_{i}|^{s_i}}\text{ in }\Omega\\
\frac{\partial u}{\partial \nu} = 0 \text{ on }\partial\Omega
\end{array}\right.
\end{align*}
where $x_i \in \overline{\Omega}$ for $1 \leq i \leq I$ and $x_{i_1} \neq x_{i_2}$ if $i_1 \neq i_2$. 

\section{Palais-Smale Condition}
In this section, we investigate the threshold of the Palais-Smale condition for $J_\lambda$. In what follows, $S_s$ denotes $S_s(\mathbb{R}^N)$. First we recall the Hardy-Sobolev inequality for functions supported on neighborhood of boundary. For the Sobolev inequality, see Lemma 2.1\ in \cite{xjwang}.
The following lemma is obtained by applying the technique of \cite{xjwang}.
\begin{lem}[Proposition 2.3 in \cite{Hashizume}]\label{HS}
Let $h(x')$ is a $C^1$ function defined in $\{x' \in \mathbb{R}^{n-1}, |x'| <1 \}$ and satisfying $\nabla h(0) = 0$. Denote $\widetilde{B} = B_1(0) \cap \{x_n > h(x')\}$. Then for any $\phi \in H^1_0(B_1(0))$, we have
\begin{enumerate}
\item If $h \equiv 0$, then
\begin{align*}
2^{\frac{2-2^*(s)}{2^*(s)}}S_s\left(\int_{\widetilde B} \frac{|\phi|^{2^*(s)}}{|x|^s}dx\right)^{\frac{2}{2^*(s)}} \leq \int_{\widetilde{B}} |\nabla \phi|^2 dx .
\end{align*}
\item For any $\ep > 0$, there exists $\delta > 0$ such that if $|\nabla h| \leq \delta$, then
\begin{align*}
(2^{\frac{2-2^*(s)}{2^*(s)}}S_s -\ep)\left(\int_{\widetilde B} \frac{|\phi|^{2^*(s)}}{|x|^s}dx\right)^{\frac{2}{2^*(s)}} \leq \int_{\widetilde{B}} |\nabla \phi|^2 dx.
\end{align*}
\end{enumerate}
\end{lem}


\begin{prop}\label{ps}
The functional $J_\lambda$ defined in \eqref{energy} satisfies the $(PS)_c$ condition for
\begin{align*}
\left\{
\begin{array}{ll}
c < \frac{2-s}{2(N-s)}S_s^{\frac{2^*(s)}{2^*(s)-2}} & \text{ if } x_1, x_2 \in \Omega\\
c < \frac{2-s}{4(N-s)}S_s^{\frac{2^*(s)}{2^*(s)-2}} & \text{ if } x_1 \text{ or } x_2 \in \partial\Omega.
\end{array}
\right.
\end{align*}
\end{prop}

\begin{proof}
The proof is based on P. L. Lions' concentration-compactness principle \cite{lions, str}. Suppose $\{u_m\}$ be a $(PS)_c$ sequence. That is
\begin{align}
J_\lambda(u_m) = \int_{\Omega} \frac12 \left( |\nabla u_m|^2 + \lambda u_m^2 \right) - \frac{1}{2^*(s)} \left(\frac{(u_{m})_+^{2^*(s)}}{|x-x_1|^s} + \frac{(u_{m})_+^{2^*(s)}}{|x-x_2|^s} \right) dx \rightarrow c \label{psc1} \\ 
\langle J_\lambda'(u_m), \phi \rangle = \int_\Omega \nabla u_m \nabla \phi + \lambda u_m \phi - \left(\frac{(u_{m})_+^{2^*(s)-1}}{|x-x_1|^s} + \frac{(u_{m})_+^{2^*(s)-1}}{|x-x_2|^s}\phi \right)dx \rightarrow 0 \label{psc2}
\end{align}
as $m \rightarrow \infty$.
Plugging $\phi = u_m$ into \eqref{psc2}, we see that
\begin{align}\label{psc3}
\int_\Omega |\nabla u_m|^2 + \lambda u_m^2 - \left(\frac{(u_{m})_+^{2^*(s)}}{|x-x_1|^s} + \frac{(u_{m})_+^{2^*(s)}}{|x-x_2|^s}\right) = o(1)\|u_m\|_{H^1}
\end{align}
Taking off one-half of \eqref{psc3} from \eqref{psc1}, we obtain
\begin{align}\label{psc4}
\left(\frac12 - \frac{1}{2^*(s)}\right) \int_\Omega \left(\frac{(u_{m})_+^{2^*(s)}}{|x-x_1|^s} + \frac{(u_{m})_+^{2^*(s)}}{|x-x_2|^s}\right)dx  \leq c + 1 + o(\|u_m\|_{H^1})
\end{align}
Hence, we derive from \eqref{psc1} that 
\begin{align*}
\frac12 \int_{\Omega}(|\nabla u_m|^2 + \lambda u_m^2)dx &\leq c + \frac{N-2}{2-s}(c+1+o(1)\|u_m\|_{H^1})\\
&\leq C(\ep) + \ep\|u_m\|_{H^1}^2.
\end{align*}
Hence $\{u_m\}$ is a bounded sequence in $H^1(\Omega)$. So, up to a subsequence, we have the following weak convergence :
\begin{align*}
u_m &\rightharpoonup u \text{ in } H^1(\Omega),\\
u_m &\rightharpoonup u \text{ in } L^{\frac{2N}{N-2}}(\Omega),\\
u_m &\rightharpoonup u \text{ in } L^{2^*(s)}(\Omega, |x-x_1|^{-s}),\\
u_m &\rightharpoonup u \text{ in } L^{2^*(s)}(\Omega, |x-x_2|^{-s}).
\end{align*}
Here $L^{2^*(s)}(\Omega, |x-x_k|^{-s}), k=1,2$ is $L^{2^*(s)}$ function space equipped with the measure $|x-x_k|^{-s}dx$.

Then the concentration-compactness principle gives
\begin{align*}
&|\nabla u_m|^2 dx \rightharpoonup d\mu \geq |\nabla u|^{2}dx + \mu_1 \delta_{x_1} + \mu_2 \delta_{x_2} + \sum_{i \in I} \mu_i \delta_{x_i},\\
&|u_m|^{\frac{2N}{N-2}} dx  \rightharpoonup |u|^{\frac{2N}{N-2}}dx + \overline\mu_1 \delta_{x_1} + \overline\mu_2 \delta_{x_2} + \sum_{i \in I} \overline\mu_i \delta_{x_i},\\
&\frac{|u_m|^{2^*(s)}}{|x-x_1|^s} dx \rightharpoonup \frac{|u|^{2^*(s)}}{|x-x_1|^s}dx + \wt\nu_1\delta_{x_1} + \wt\nu_2\delta_{x_2} + \sum_{i \in I} \wt\nu_i \delta_{x_i},\\ 
&\frac{|u_m|^{2^*(s)}}{|x-x_2|^s} dx \rightharpoonup \frac{|u|^{2^*(s)}}{|x-x_2|^s}dx + \overline\nu_1\delta_{x_1} + \overline\nu_2\delta_{x_2} + \sum_{i \in I} \overline\nu_i \delta_{x_i}
\end{align*}
in the sense of measure where $\delta_x$ is the Dirac-mass of mass 1 concentrated at $x \in \mathbb{R}^N$. Here, $I$ is at most countable index set and the numbers $\mu_i, \overline{\mu}_i, \wt\nu_i, \overline{\nu}_i \geq 0$.

We will analyse $\mu_i, \overline{\mu}_i, \wt\nu_i, \overline{\nu}_i$ to show that all of them is $0$.  Let $\phi$ be $C^1$ function such that $\phi(x) = 1$ on $B_{1}(0)$ and $\phi(x) = 0$ on $\mathbb{R}^N \setminus B_2(0)$. We define $\phi^l(x) = \phi(lx)$. Fix $l >0$.
Then for $k=1$ or $2$, we have from weak convergence
\begin{align*}
&\quad \int_{\Omega} \frac{|u_m|^{2^*(s)}}{|x-x_k|^s}(1 - \phi^l(\cdot - x_k))dx\\
&\rightarrow \int_{\Omega} \frac{|u|^{2^*(s)}}{|x-x_k|^s}(1 - \phi^l(\cdot - x_k))dx + \sum_{x_i \in \mathbb{R}^N \setminus B_{\frac 2l}(x_k)} \widehat \nu_i (1 - \phi^l(x_i))
\end{align*}
as $m \rightarrow \infty$, where $\widehat \nu_i  = \wt\nu_i$ or $\overline\nu_i$ when $k=1$ or $k=2$, respectively,
Since $2^*(s) < \frac{2N}{N-2}$, we have from strong convergence
\[\int_{\Omega} \frac{|u_m|^{2^*(s)}}{|x-x_k|^s}(1 - \phi^l(\cdot - x_k))dx \rightarrow \int_{\Omega} \frac{|u|^{2^*(s)}}{|x-x_k|^s}(1 - \phi^l(\cdot - x_k))dx,\]
as $m \rightarrow \infty$. So we obtain
\begin{align*}
\left\{
\begin{array}{c}
\wt\nu_i = 0 \text{ if } x_i \in \mathbb{R}^N \setminus B_{\frac 2l}(x_1),\\
\cj\nu_i = 0\text{ if } x_i \in \mathbb{R}^N \setminus B_{\frac 2l}(x_2).
\end{array}
\right.
\end{align*}
Letting $l \rightarrow \infty$, we see that
\[\frac{|u_m|^{2^*(s)}}{|x-x_1|^s}dx \rightharpoonup \frac{|u|^{2^*(s)}}{|x-x_1|^s}dx + \nu_1\delta_{x_1} \text{ and }\frac{|u_m|^{2^*(s)}}{|x-x_2|^s}dx \rightharpoonup \frac{|u|^{2^*(s)}}{|x-x_2|^s}dx + \nu_2\delta_{x_2}\] 
where $\nu_1 : = \wt\nu_1$ and $\nu_2 := \cj\nu_2$.

Now we shall show some relation between $\nu_k$ and $\mu_k$ for $k=1,2$. We consider $v_m = u_m - u$ and
\[d\omega_m := \left( \frac{|u_m|^{2^*(s)}}{|x-x_k|^s} - \frac{|u|^{2^*(s)}}{|x-x_k|^s}\right) dx = \frac{|u_m-u|
^{2^*(s)}}{|x-x_k|^s}dx + o(1).\]
In case of $x_k \in \Omega$, we have
\begin{align*}
\int_{\Omega}|\phi^l(\cdot-x_k)|^{2^*(s)}d\omega_m &= \int_{\Omega} \frac{|\phi^l(\cdot-x_k)v_m|^{2^*(s)}}{|x-x_k|^s}dx + o(1)\\
&\leq S_s^{-\frac{2^*(s)}{2}}(\int_{\Omega}|\nabla(\phi^l(\cdot-x_k)v_m)|^2dx)^{\frac{2^*(s)}{2}}+ o(1)
\end{align*}
For fixed $l$, we see that $\phi^l, \nabla \phi^l \in L^\infty(\Omega)$. Moreover, since $\nabla v_m \rightarrow 0$ weakly in $L^2$, we have $v_m \rightarrow 0 $ in $L^p$ for $0 < p < \frac{2N}{N-2}$ by Rellich-Kondrakov Theorem. So we get
\begin{align*}
&\quad \int_\Omega |\nabla(\phi^l(\cdot - x_k)v_m)|^2 dx\\
&\leq \int_\Omega |\nabla(\phi^l(\cdot - x_k))|^2|v_m|^2 dx + C_l\left(\int_\Omega |v_m|^2 dx\right)^\frac12 \left(\int_\Omega|\nabla v_m|^2dx\right)^\frac12\\
&\quad + \int_\Omega |\phi^l(\cdot - x_k)|^2|\nabla v_m|^2 dx\\
&= \int_\Omega |\phi^l(\cdot - x_k)|^2|\nabla v_m|^2 dx + o(1)
\end{align*}
Hence, 
\begin{align*}
\int_{\Omega}|\phi^l(\cdot-x_k)|^{2^*(s)}d\omega_m  &= S_s^{-\frac{2^*(s)}{2}}(\int_{\Omega}(\phi^l(\cdot-x_k))^2|\nabla v_m|^2dx)^{\frac{2^*(s)}{2}}+ o(1).
\end{align*}
In case of $x_k \in \partial\Omega$, applying Lemma \ref{HS}, we see that
\begin{align*}
&\quad \int_{\Omega}|\phi^l(\cdot-x_k)|^{2^*(s)}d\omega_m\\
&= \int_{\Omega} \frac{|\phi^l(\cdot-x_k)v_m|^{2^*(s)}}{|x-x_k|^s}dx+ o(1)\\
&\leq (2^{\frac{2^*(s) - 2}{2}}S_s^{-\frac{2^*(s)}{2}} + \ep_l)\left(\int_{\Omega}|\nabla(\phi^l(\cdot-x_k)v_m)|^2dx\right)^{\frac{2^*(s)}{2}}+ o(1)\\
&= (2^{\frac{2^*(s) - 2}{2}}S_s^{-\frac{2^*(s)}{2}} + \ep_l)\left(\int_{\Omega}(\phi^l(\cdot-x_k))^2|\nabla v_m|^2dx\right)^{\frac{2^*(s)}{2}}+ o(1)
\end{align*}
where $\ep_l \rightarrow 0$ as $l \rightarrow \infty.$
By letting $m,l \rightarrow \infty$, we obtain
\begin{align}\label{ssmunu}
\left\{\begin{array}{cl}
S_s \nu_k^{\frac{2}{2^*(s)}} \leq \mu_k &\text{ if } x_k \in \Omega\\
2^{\frac{2 - 2^*(s)}{2^*(s)}}S_s \nu_k^{\frac{2}{2^*(s)}} \leq \mu_k &\text{ if } x_k \in \partial \Omega
\end{array}
\right.
\end{align}
for $k=1,2$.

To complete the proof, we need to show that $\mu_i = 0$ for $i=1,2$ or $i \in I$.
For $i \in I$, by testing $u_m(x)\phi^l(x-x_i)$, we have
\begin{align*}
\langle J'_\lambda(u_m), u_m\phi^l(\cdot - x_i)\rangle &=\int_\Omega \nabla u_m \nabla (u_m\phi^l(\cdot-x_i)) + \lambda u_m u_m\phi^l(\cdot-x_i)\\
&\quad - \left(\frac{(u_m)_+^{2^*(s)}}{|x-x_1|^s}\phi^l(\cdot-x_i)+\frac{(u_m)_+^{2^*(s)}}{|x-x_2|^s}\phi^l(\cdot-x_i)\right).
\end{align*}
One can readily check that
\begin{align*}
&\lim_{l \rightarrow \infty}\lim_{m \rightarrow \infty} \int_\Omega \nabla u_m \nabla u_m\phi^l(\cdot-x_i) \geq \mu_i,\\
&\lim_{l \rightarrow \infty}\lim_{m \rightarrow \infty} \int_\Omega \lambda u_m u_m\phi^l(\cdot-x_i) = 0,\\
&\lim_{l \rightarrow \infty}\lim_{m \rightarrow \infty} \int_\Omega \frac{(u_m)_+^{2^*(s)}}{|x-x_1|^s}\phi^l(\cdot-x_i)=0,\\
&\lim_{l \rightarrow \infty}\lim_{m \rightarrow \infty} \int_\Omega \frac{(u_m)_+^{2^*(s)}}{|x-x_2|^s}\phi^l(\cdot-x_i) = 0.
\end{align*}
We claim that
\[\lim_{l \rightarrow \infty}\lim_{m \rightarrow \infty}\int_\Omega \nabla u_m u_m \nabla\phi^l(\cdot-x_i) = 0.\]
Let $\Omega^l_i := \Omega \cap supp(\nabla\phi^l(\cdot-x_i))$. First we consider the case where $x_i$ is not a limit point of $\{x_k : k \in I\}$. In this case, we see that
\[x_i \notin \Omega^l_i \text{ for all }l\]
and 
\[x_k \notin \Omega^l_i \text{ for }k \neq i \text{ as } l \text{ is sufficiently large}.\]
Hence we have
\begin{align*}
&\quad \left|\lim_{l \rightarrow \infty}\lim_{m \rightarrow \infty}\int_\Omega \nabla u_m u_m \nabla\phi^l(\cdot-x_i)\right|\\
&=\lim_{l \rightarrow \infty}\left|\int_{\Omega}\nabla u \cdot \nabla \phi^l(\cdot - x_i)u dx\right|\\
&\leq \lim_{l \rightarrow \infty}\left(\int_{\Omega^l_j}|\nabla u|^2dx\right)^\frac12\left(\int_{\Omega^l_j}|u|^{\frac{2N}{N-2}}dx\right)^\frac{N-2}{2N}\left(\int_{\Omega^l_j} |\nabla\phi^l(\cdot-x_i)|^{N}dx\right)^\frac1N \\
&\leq \lim_{l \rightarrow \infty} C\left(\int_{\Omega^l_j}|\nabla u|^2dx\right)^\frac12\left(\int_{\Omega^l_j}|u|^{\frac{2N}{N-2}}dx\right)^\frac{N-2}{2N}\\
&=0.
\end{align*}
In the case of $x_i$ is a limit point of $\{x_k : k \in I\}$, there is additional term
\[(\sum_{k \in I, x_k \in \Omega^l_j} \mu_k)^{\frac12}(\sum_{k \in I, x_k \in \Omega^l_j} \overline{\mu}_k)^{\frac{N-2}{2N}} \leq \left(\int_{\Omega^l_j}|\nabla u|^2dx\right)^\frac12\left(\int_{\Omega^l_j}|u|^{\frac{2N}{N-2}}dx\right)^\frac{N-2}{2N} < \infty\]
which also goes to $0$ as $l \rightarrow \infty.$
So we get
\begin{align*}
\mu_i \leq  \lim_{l \rightarrow \infty} \langle J'_\lambda(u_m), u_m\phi^l(\cdot - x_i)\rangle = \lim_{l \rightarrow \infty} o(\|u_m\phi^l\|_{H^1}) = 0 \text{ for }  i \in I.
\end{align*}

Using the same argument, we have for $i = 1,2$,
\begin{align}\label{munu}
\mu_i = \lim_{l \rightarrow \infty}\lim_{m \rightarrow \infty} \int_\Omega \nabla u_m \nabla u_m\phi^l(\cdot-x_i)
=\lim_{l \rightarrow \infty}\lim_{m \rightarrow \infty} \int_\Omega \frac{(u_m)_+^{2^*(s)}}{|x-x_i|^s}\phi^l(\cdot-x_i)\leq \nu_i
\end{align}
If we assume $\mu_i > 0$ for $i=1$ or $2$, then
\[c = \lim_{m \rightarrow \infty } J_\lambda(u_m) - \frac12\langle J'_\lambda(u_m), u_m \rangle \geq (\frac12 - \frac{1}{2^*(s)})\mu_i.\]
But from \eqref{ssmunu} and \eqref{munu}, we have
\begin{align*}
\left\{\begin{array}{cl}
S_s \mu_i^{\frac{2}{2^*(s)}} \leq \mu_i \Leftrightarrow \mu_i \geq S_s^{\frac{2^*(s)}{2^*(s)-2}}\\
2^{\frac{2 - 2^*(s)}{2^*(s)}}S_s \mu_i^{\frac{2}{2^*(s)}} \leq \mu_i \Leftrightarrow \mu_i \geq \frac12 S_s^{\frac{2^*(s)}{2^*(s)-2}}
\end{array}
\right.
\end{align*}
which is a contradiction. This prove Proposition \ref{ps}.
\end{proof}

\section{Existence of solution to \eqref{main in} for small $\lambda$}
In this section, we show the existence theory of Theorem \ref{thm1}.
Plugging constant function $c$ into the functional $J_\lambda$, we have
\[J_\lambda(c) = \frac12 |\Omega| \lambda c^2 - \frac{1}{2^*(s)}C_1c^{2^*(s)}\]
where $C_1 = \int_\Omega \frac{1}{|x-x_1|^s}+\frac{1}{|x-x_2|^s} dx$.
Since $2^*(s) > 2$, we see that $J_\lambda(c) < 0$ for sufficiently large $c$.
From the observation
\[\frac{d}{dc}\left(J_\lambda(c)\right) = 0 \Leftrightarrow c=0 \text{ or } \left(\frac{\lambda|\Omega|}{C_1}\right)^{\frac{1}{2^*(s)-2}},\]
we see that
\begin{align*}
\max_{c}J_\lambda (c) &= J_\lambda\left((\frac{\lambda|\Omega|}{C_1})^{\frac{1}{2^*(s)-2}}\right)\\
&= \frac12|\Omega|\lambda\left(\frac{\lambda|\Omega|}{C_1}\right)^{\frac{2}{2^*(s)-2}} - \frac{1}{2^*(s)}C_1\left(\frac{\lambda|\Omega|}{C_1}\right)^{\frac{2^*(s)}{2^*(s)-2}}
\end{align*}
which is less than
\begin{align*}
\left\{
\begin{array}{ll}
\frac{2-s}{2(N-s)}S_s^{\frac{2^*(s)}{2^*(s)-2}} & \text{ if } x_1, x_2 \in \Omega\\
\frac{2-s}{4(N-s)}S_s^{\frac{2^*(s)}{2^*(s)-2}} & \text{ if } x_1 \text{ or } x_2 \in \partial\Omega
\end{array}
\right.
\end{align*}
provided the positive solution parameter $\lambda$ is small enough.

\section{Existence of solution to \eqref{main in} with boundary singularity}
In this section, we prove the existence of a solution in Theorem \ref{thm2}. We shall follow the strategy of \cite{xjwang, ghka} to prove Theorem \ref{thm2}. We may assume $x_1 = (0, \cdots, 0) \in \partial \Omega$ and the mean curvature $H(0)$ is positive.
Then, up to rotation, the boundary near the origin can be represented by 
\[x_n = h(x') = \frac12 \sum_{i=1}^{N-1}\al_i x_i^2 + o(|x'|^2)\]
where $x' = (x_1, x_2, \cdots, x_{N-1}) \in D_\delta(0) = B_\delta(0) \cap \{x_N=0\}$ for some $\delta > 0$. Here $\al_1, \al_2, \cdots, \al_{N-1}$ are the principal curvature of $\partial \Omega$ at $0$ and the mean curvature $\sum_{i=1}^{N-1} \al_i > 0$. Denote
\[g(x') = \frac12\sum_{i=1}^{N-1}\al_ix_i^2.\] 

Consider
\begin{equation*}
U_\ep(x):=\ep^{\frac{N-2}{2(2-s)}}(\ep+|x|^{2-s})^{\frac{2-N}{2-s}}
\end{equation*}
for small parameter $\e>0$.
Then,
it follows that
\begin{equation}\label{musnorm}
\int_{\re^N} |\nabla U_1|^2dx / (\int_{\re^N} \frac{|U_1|^{2^*(s)}}{|x|^s}dx)^{\frac{N-2}{N-s}} = S_s.
\end{equation}
Choose $\delta$ such that $x_2 \notin B_{3\delta}(0)$. Set a cut-off function $\eta$ such that
\begin{equation*}
\quad \eta \in C_c^\infty(\re^N), \quad 0 \leq \eta \leq 1, \quad \eta =1 \ {\rm in} \ B_{\delta}(0), \quad\eta = 0 \ {\rm in}\ \re^N\setminus{B_{2\delta}(0)}.
\end{equation*}
Note that from $2^*(s) >2$,
\begin{align*}
J_\lambda(T\eta U_\ep) &= \int_{\Omega}\frac{T^2}{2}(|\nabla (\eta U_\ep)|^2 + \lambda (\eta U_\ep)^2) - \frac{T^{2^*(s)}}{2^*(s)}\left(\frac{(\eta U_\ep)_+^{2^*(s)}}{|x-x_1|^s}+\frac{(\eta U_\ep)_+^{2^*(s)}}{|x-x_2|^s}\right)dx\\
 &< 0
\end{align*}
for sufficiently large $T$.
We define
\begin{align*}
\mathbb{P} =
\left\{p(t) \Big|
\begin{array}{ll}
 p(t) : [0,1] \rightarrow H^1(\Omega)\text{ is continous map with }\\
p(0) = 0 \in H^1(\Omega)\text{ and }p(1) = T\eta U_\ep(x)|_\Omega
\end{array}
\right\}.
\end{align*}
Let 
\[c^* = \inf_{p(t) \in \mathbb{P}}\sup_{0 \leq t \leq 1} \{ J_\lambda(p(t)) \}.\]
Then, thanks to Proposition \ref{ps}, it suffices to show 
\begin{align}\label{threshold}
c^* < \frac{2-s}{4(N-s)} S_s^{\frac{2^*(s)}{2^*(s)-2}}.
\end{align}
In the following discussion, we denote
\begin{align*}
K_0^\ep := \int_{\Omega} |\nabla(\eta U_\ep)|^2 dx,\ &K_1^\ep := \int_\Omega \frac{(\eta U_\ep)^{2^*(s)}}{|x - x_1|^s}dx, K_3^\ep := \int_\Omega (\eta U_\ep)^2dx\\ &\text{ and } K_2^\ep := \int_\Omega \frac{(\eta U_\ep)^{2^*(s)}}{|x - x_2|^s}dx.
\end{align*}

First we deal with $K_0^\ep$. By using Leibniz rule, one has
\begin{align*}
K_0^\ep &= \int_{\Omega} |\nabla(\eta U_\ep)|^2 dx\\
&= \int_{\Omega} |\nabla \eta|^2 |U_\ep|^2 dx + 2\int_\Omega \eta U_\ep \nabla \eta \cdot\nabla U_\ep dx + \int_{\Omega} |\eta|^2 |\nabla U_\ep|^2 dx.
\end{align*}
Since $supp (\nabla \eta) \subset B_{2\delta}(0) \setminus B_{\delta}(0)$, for sufficiently small $\ep$,
\begin{align*}
\int_\Omega |\nabla \eta|^2 |U_\ep|^2 dx &= \int_\Omega |\nabla \eta|^2 \ep^{\frac{N-2}{2-s}}(\ep + |x|^{2-s})^{\frac{2(2-N)}{2-s}}dx\\
&\leq \int_\Omega |\nabla \eta|^2 \ep^{\frac{N-2}{2-s}} (2\delta)^{{2(2-N)}}dx\\
&\leq C_{1,\delta} \ep^{\frac{N-2}{2-s}}.
\end{align*}
Similarly, it follows that
\begin{align*}
&\quad \left|\int_\Omega \eta  U_\ep \nabla \eta \cdot \nabla U_\ep dx\right|\\
&= \Big|\int_\Omega \eta \ep^{\frac{N-2}{2(2-s)}}(\ep + |x|^{2-s})^{\frac{2-N}{2-s}}\\
&\qquad\qquad \times \nabla \eta \cdot \left((2-N)\ep^{\frac{N-2}{2(2-s)}}(\ep+|x|^{2-s})^{\frac{2-N}{2-s}-1}|x|^{1-s}\frac{x}{|x|}\right)dx\Big|\\
&\leq C_{2,\delta} \ep^{\frac{N-2}{2-s}}.
\end{align*}
The last term is more delicate. We consider the case of $N=3$ and the case of $N \geq 4$ separately.
When $N = 3$, we have
\begin{align*}
\int_\Omega |\eta|^2|\nabla U_\ep|^2 dx &= \int_{\mathbb{R}^N_+}|\nabla U_\ep|^2dx - \int_{D_\delta(0)}\int_0^{h(x')}|\nabla U_\ep|^2 dx_Ndx' + o(\ep^{\frac{1}{2-s}}).
\end{align*}
Since $a|x'|^2 \leq h(x') \leq A|x'|^2$ on $D_\delta(0)$ for some $0 < a \leq A < \infty$, we have
\begin{align*}
\int_{D_\delta(0)}\int_0^{h(x')}|\nabla U_\ep|^2 dx_Ndx' &\geq C \int_{D_\delta(0)} \frac{\ep^{\frac{N-2}{2-s}}a|x'|^{4-2s}}{(\ep + |x'|^{2-s})^\frac{2(N-s)}{2-s}} dx'\\
&\geq C\ep^{\frac{1}{2-s}}|\ln\ep|.
\end{align*}
When $N \geq 4$, we have
\begin{align*}
&\quad \int_\Omega |\eta|^2|\nabla U_\ep|^2 dx\\ &= \int_{\mathbb{R}^N_+}|\nabla U_\ep|^2dx - \int_{D_\delta(0)}\int_0^{h(x')}|\nabla U_\ep|^2 dx_Ndx' + O(\ep^{\frac{N-2}{2-s}})\\
&=\frac12 K_0 - \int_{\mathbb{R}^{N-1}}\int_0^{g(x')}|\nabla U_\ep|^2 dx_Ndx' - \int_{D_\delta(0)}\int_{g(x')}^{h(x')}|\nabla U_\ep|^2 dx_Ndx' + O(\ep^{\frac{N-2}{2-s}})
\end{align*}
where
\[K_0 := \int_{\mathbb{R}^N} |\nabla U_\ep|^2dx = (N-2)^2 \int_{\mathbb{R}^N} \frac{|y|^{(2-2s)}}{(1+|y|^{2-s})^{\frac{2(N-s)}{2-s}}}dy.\]
Observe that
\begin{align}\label{iep}
\begin{split}
I(\ep)&:= \int_{\mathbb{R}^{N-1}} \int_0^{g(x')} |\nabla U_\ep|^2 dx_Ndx' \\
&=(N-2)^2\ep^{\frac{N-2}{2-s}}\int_{\mathbb{R}^{N-1}}\int_0^{g(x')} \frac{|x|^{2-2s}}{(\ep+|x|^{2-s})^{\frac{2(N-s)}{2-s}}}dx_Ndx'\\
&=(N-2)^2\int_{\mathbb{R}^{N-1}}\int_0^{g(y')\ep^{\frac{1}{2-s}}} \frac{|y|^{2-2s}}{(1+|y|^{2-s})^{\frac{2(N-s)}{2-s}}}dy_Ndy'.
\end{split}
\end{align}
So we have 
\begin{align*}
\lim_{\ep \rightarrow 0}\ep^{-\frac{1}{2-s}}I(\ep) &= (N-2)^2 \int_{\mathbb{R}^{N-1}}\frac{|x'|^{2-2s}g(x')}{(1+|x'|^{2-s})^{\frac{2(N-s)}{2-s}}}dx'\\
&= \frac{(N-2)^2}{2}\int_{\mathbb{R}^{N-1}}\frac{|x'|^{2-2s}\sum_{i=1}^{N-1}\al_i|x_i|^2}{(1+|x'|^{2-s})^{\frac{2(N-s)}{2-s}}}dx'\\
&= \frac{(N-2)^2}{2}\sum_{i=1}^{N-1}\al_i\int_{\mathbb{R}^{N-1}}\frac{|x'|^{2-2s}|x_i|^2}{(1+|x'|^{2-s})^{\frac{2(N-s)}{2-s}}}dx'\\
&= (\sum_{i=1}^{N-1}\al_i)\frac{(N-2)^2}{2(N-1)}\int_{\mathbb{R}^{N-1}}\frac{|x'|^{4-2s}}{(1+|x'|^{2-s})^{\frac{2(N-s)}{2-s}}}dx'.
\end{align*}
which leads to
\begin{align*}
I(\ep) = O(\ep^{\frac{1}{2-s}}).
\end{align*}
The curvature assumption ($H(0) > 0$) implies
\[I(\ep) > 0.\]
Moreover,
\begin{align*}
I_1(\ep) &:= \int_{D_\delta(0)}\int_{g(x')}^{h(x')}|\nabla U_\ep|^2 dx_ndx'\\
&= (N-2)^2 \ep^{\frac{N-2}{2-s}}\int_{D_\delta(0)}\int_{g(x')}^{h(x')}\frac{|x|^{2-2s}}{(\ep + |x|^{2-2s})^{\frac{2(N-s)}{2-s}}}dx_Ndx'\\
&\leq C(\delta, N)(N-2)^2 \ep^{\frac{N-2}{2-s}} \int_{D_\delta(0)} \frac{|h(x') - g(x')|}{(\ep+|x'|^{2-s})^{\frac{2(N-s)}{2-s}-1}}dx'
\end{align*}
Since $h(x') = g(x') + o(|x'|^2)$, for any $\sigma > 0$, there exists $C(\sigma) > 0$ such that
\[|h(x') - g(x')| \leq \sigma|x'|^2 + C(\sigma)|x'|^{\frac52}.\]
So we have
\begin{align*}
I_1(\ep) &\leq C\ep^{\frac{N-2}{2-s}} \int_{D_\delta(0)} \frac{\sigma|x'|^2 + C(\sigma)|x'|^\frac52}{(\ep+|x'|^{2-s})^{\frac{2(N-s)}{2-s}-1}}dx' \\
&\leq C\ep^{\frac{1}{2-s}}(\sigma + C(\sigma)\ep^{\frac{1}{2(2-s)}})
\end{align*}
which implies
\begin{align*}
I_1(\ep)&= O(\ep^{\frac{1}{2-s}}).
\end{align*}
Therefore we obtain
\begin{align}\label{k1ep}
K_0^\ep = 
\left\{\begin{array}{ll}
\frac12 K_0 - C\ep^{\frac{1}{2-s}}|\ln\ep| + O(\ep^{\frac{1}{2-s}}) &\text{ when } N = 3,\\
\frac12 K_0 - I(\ep) + O(\ep^{\frac{1}{2-s}}) &\text{ when } N \geq 4.
\end{array}
\right.
\end{align}
On the other hand, we have
\begin{align*}
K_1^\ep &= \int_{\mathbb{R}^N_+} \frac{|U_\ep|^{2^*(s)}}{|x|^s}dx - \int_{D_\delta(0)}\int_0^{h(x')}\frac{|U_\ep|^{2^*(s)}}{|x|^s}dx_Ndx' + O(\ep^\frac{N-s}{2-s})\\
&= \frac12 K_1 - \int_{\mathbb{R}^{N-1}}\int_0^{g(x')}\frac{|U_\ep|^{2^*(s)}}{|x|^s}dx_Ndx'\\ 
&\quad - \int_{D_\delta(0)}\int_{g(x')}^{h(x')}\frac{|U_\ep|^{2^*(s)}}{|x|^s}dx_Ndx' + O(\ep^\frac{N-s}{2-s}) 
\end{align*}
where
\begin{align*}
K_1 &= \int_{\mathbb{R}^N} \frac{U_\ep^{2^*(s)}}{|x|^s}dx = \int_{\mathbb{R}^N} \frac{\ep^{\frac{2^*(s)(N-2)}{2(2-s)}}}{|x|^s(\ep + |x|^{2-s})^{\frac{2^*(s)(N-2)}{2-s}}}dx\\
&= \int_{\mathbb{R}^N}\frac{1}{|y|^s(1+|y|^{2-s})^{\frac{2(N-s)}{2-s}}}dy.
\end{align*}
Observe that
\begin{align}\label{iiep}
\begin{split}
I\!I(\ep) &:= \int_{\mathbb{R}^{N-1}} \int_0^{g(x')} \frac{|U_\ep|^{2^*(s)}}{|x|^s}dx_Ndx' \\
&= \int_{\mathbb{R}^{N-1}}\int_0^{\ep^{\frac{1}{2-s}}g(y')} \frac{1}{|y|^s(1+|y|^{2-s})^{\frac{2(N-s)}{2-s}}}dy_Ndy'.
\end{split}
\end{align}
So, we have
\begin{align*}
\lim_{\ep \rightarrow 0} \ep^{-\frac{1}{2-s}}I\!I(\ep) &= \int_{\mathbb{R}^{N-1}} \frac{g(y')}{|y'|^s(1+|y'|^{2-s})^{\frac{2(N-s)}{2-s}}}dy'\\
&=\frac12 \int_{\mathbb{R}^{N-1}} \frac{\sum_{i=1}^{N-1}\al_i|y_i|^2}{|y'|^s(1+|y'|^{2-s})^{\frac{2(N-s)}{2-s}}}dy'\\
&=\frac{\sum_{i=1}^{N-1}\al_i}{2(N-1)}\int_{\mathbb{R}^{N-1}} \frac{|y'|^2}{|y'|^s(1+|y'|^{2-s})^{\frac{2(N-s)}{2-s}}}dy'
\end{align*}
which leads to
\[I\!I(\ep) = O(\ep^{\frac{1}{2-s}}).\]
The curvature assumption ($H(0) > 0$) implies
\[I\!I(\ep) > 0.\]
Similarly, we can get
\[\int_{D_\delta(0)}\int_{g(x')}^{h(x')} \frac{|U_\ep|^{2^*(s)}}{|x|^s}dx_Ndx' = O(\ep^\frac{1}{2-s}).\]
Thus, we obtain
\begin{align}\label{k2ep}
K_1^\ep = \frac12 K_1 - I\!I(\ep) + O(\ep^{\frac{1}{2-s}}).
\end{align}

Moreover, direct calculation gives
\begin{align*}
K_3^\ep =\int_\Omega (\eta U_\ep)^2 dx =
\left\{\begin{array}{ll}
O(\ep^{\frac{1}{2-s}}), & N=3,\\
O(|\ep^{\frac{2}{2-s}}\ln \ep|), & N=4,\\
O(\ep^{\frac{2}{2-s}}), & N \geq 5.
\end{array}
\right.
\end{align*}
Actually when $N = 3$, we have
\begin{align*}
\int_\Omega (\eta U_\ep)^2 dx &= \int_\Omega |\eta|^2 \ep^{\frac{1}{2-s}}(\ep + |x|^{2-s})^{-\frac{2}{2-s}}dx\\
&\leq \int_{B_{2\delta(0)}} \ep^{\frac{1}{2-s}}(\ep + |x|^{2-s})^{-\frac{2}{2-s}}dx \\
&\leq C_N\ep^{\frac{2}{2-s}} \int_0^{2\delta\ep^{-\frac{1}{2-s}}} (1 + r^{2-s})^{-\frac{2}{2-s}}r^{2}dr = O(\ep^{\frac{1}{2-s}}).
\end{align*}
When $N = 4$, we see that
\begin{align*}
\int_\Omega (\eta U_\ep)^2 dx &= \int_\Omega |\eta|^2 \ep^{\frac{2}{2-s}}(\ep + |x|^{2-s})^{-\frac{4}{2-s}}dx\\
&\leq \int_{B_{2\delta(0)}} \ep^{\frac{2}{2-s}}(\ep + |x|^{2-s})^{-\frac{4}{2-s}}dx \\
&\leq C_N\ep^{\frac{2}{2-s}} \int_0^{2\delta\ep^{-\frac{1}{2-s}}} (1 + r^{2-s})^{-\frac{4}{2-s}}r^{N-1}dr = O(|\ep^{\frac{2}{2-s}}\ln\ep|).
\end{align*}
When $N \geq 5$, we have
\begin{align*}
\int_\Omega (\eta U_\ep)^2 dx &= \int_\Omega |\eta|^2 \ep^{\frac{N-2}{2-s}}(\ep + |x|^{2-s})^{\frac{2(2-N)}{2-s}}dx\\
&\leq \int_{\mathbb{R}^N} \ep^{\frac{N-2}{2-s}}(\ep + |x|^{2-s})^{\frac{2(2-N)}{2-s}}dx = O(\ep^{\frac{2}{2-s}}).
\end{align*}

Lastly, we are concerned about $K_2^\ep$. Since $x_2 \notin B_{3\delta}(0)$ and $supp(\eta) \subset B_{2\delta}(0)$, we see that
\begin{align*}
&\quad K^\ep_2=\int_\Omega \frac{|\eta U_\e|^{2^*(s)}}{|x - x_2|^s} dx\\ &\leq C\int_{\Omega \cap B_{2\delta(0)}}|U_\e|^{2^*(s)} dx \\ &\leq \int_{B_{2\delta(0)}} (\ep^{\frac{N-2}{2(2-s)}}(\ep+|x|^{2-s})^{\frac{2-N}{2-s}})^{\frac{2(N-s)}{N-2}}dx\\
&= \ep^{\frac{s}{2-s}} \int_{B_{2\delta\ep^{-1/(2-s)}}(0)} (1+|y|^{2-s})^{-\frac{2(N-s)}{2-s}}dy = O(\ep^{\frac{s}{2-s}}).
\end{align*}

Let $t_\ep$ be a constant satisfying
\begin{align*}
J_\lambda(t_\ep\eta U_\ep) &= \sup_{t > 0}J_\lambda(t\eta U_\ep)\\ &= \sup_{t > 0} \left[ \frac12 K_0^\ep t^2 - \frac{1}{2^*(s)}K_1^\ep t^{2^*(s)} + \frac{\lambda}{2}K_3^\ep t^2 - \frac{1}{2^*(s)}K_2^\ep t^{2^*(s)}\right]\\
&\leq \sup_{t > 0} \left[ \frac12 K_0^\ep t^2 - \frac{1}{2^*(s)}K_1^\ep t^{2^*(s)} + \frac{\lambda}{2}K_3^\ep t^2\right].
\end{align*}
In case $N = 3$, we see that $K_3(\ep) = O(\ep^{\frac{1}{2-s}})$.
Hence,
\begin{align*}
J_\lambda(t_\ep\eta U_\ep) &\leq \sup_{t>0} [\frac12 K_0^\ep t^2 - \frac{1}{2^*(s)}K_1^\ep t^{2^*(s)}] + O(\ep^{\frac{1}{2-s}})\\ 
&= \frac{2-s}{2(N-s)} \left[\frac{K_0^\ep}{(K_1^\ep)^{\frac{N-2}{N-s}}}\right]^{\frac{N-s}{2-s}} + O(\ep^{\frac{1}{2-s}}).
\end{align*}
So to prove \eqref{threshold}, it suffices to show that
\begin{align}\label{extrme1}
K_0^\ep/(K_1^\ep)^{\frac{N-2}{N-s}} < 2^{-\frac{2-s}{N-s}}S_s + O(\ep^{\frac{1}{2-s}}) = \frac12 K_0/(\frac12 K_1)^{\frac{N-2}{N-s}} + O(\ep^{\frac{1}{2-s}}).
\end{align}
Taking \eqref{musnorm}, \eqref{k1ep} and \eqref{k2ep} into account, \eqref{extrme1} is equivalent to
\begin{align*}
\frac12K_0 - C\ep^{\frac{1}{2-s}}|\ln\ep| &< 2^{-\frac{2-s}{N-s}}S_s\left[\frac12 K_1 - O(\ep^{\frac{1}{2-s}})\right]^{\frac{N-2}{N-s}} + O(\ep^{\frac{1}{2-s}})\\ 
&= \frac12S_s K_1^{\frac{N-2}{N-s}} + O(\ep^{\frac{1}{2-s}})
\end{align*}
which is true for small $\ep > 0$, because $C>0$ and
\[K_0/K_1^{\frac{N-2}{N-s}} = S_s.\]

In case $N \geq 4$, we know that $K_3^\ep = O(\ep^{\frac{1}{2-s}})$.
Hence, 
\begin{align*}
J_\lambda(t_\ep\eta U_\ep) &\leq \sup_{t>0} [\frac12 K_0^\ep t^2 - \frac{1}{2^*(s)}K_1^\ep t^{2^*(s)}] + O(\ep^{\frac{1}{2-s}})\\ &= \frac{2-s}{2(N-s)} \left[\frac{K_0^\ep}{(K_1^\ep)^{\frac{N-2}{N-s}}}\right]^{\frac{N-s}{2-s}} + O(\ep^{\frac{1}{2-s}}).
\end{align*}
So to prove \eqref{threshold}, it suffices to show that
\begin{align}\label{extrme2}
K_0^\ep/(K_1^\ep)^{\frac{N-2}{N-s}} < 2^{-\frac{2-s}{N-s}}S_s + O(\ep^{\frac{1}{2-s}}) = \frac12 K_0/(\frac12 K_1)^{\frac{N-2}{N-s}} +  O(\ep^{\frac{1}{2-s}}).
\end{align}
Taking \eqref{musnorm}, \eqref{k1ep} and \eqref{k2ep} into account, \eqref{extrme2} is equivalent to
\begin{equation}\label{K1K2ineq}
\begin{aligned}
&\quad \left(\frac12K_0 - I(\ep)\right)\left(\frac12 K_1\right)^{\frac{N-2}{N-s}}\\
&< \frac12 K_0\left(\frac12 K_1 - I\!I(\ep) + O(\ep^{\frac{\wt s}{2-s}})\right)^\frac{N-2}{N-s} + O(\ep^{\frac{1}{2-s}})\\
&= \frac12 K_0 \left\{\left(\frac12 K_1\right)^{\frac{N-2}{N-s}}  - \frac{N-2}{N-s}\left(\frac12 K_1\right)^{\frac{s-2}{N-s}}I\!I(\ep)\right\} + O(\ep^\frac{1}{2-s}).
\end{aligned}
\end{equation}
Hence to verify \eqref{K1K2ineq}, we have to prove
\[\lim_{\ep \rightarrow 0} \frac{I\!I(\ep)}{I(\ep)} < \frac{(N-s)K_1}{(N-2)K_0}.\]
By \eqref{iep}, \eqref{iiep} and L'H\^opital's rule, we obtain
\begin{align*}
&\quad \lim_{\ep \rightarrow 0} \frac{I\!I(\ep)}{I(\ep)} = \lim_{\ep \rightarrow 0} \frac{I\!I'(\ep)}{I'(\ep)}\\
&= (N-2)^{-2} \int_{\mathbb{R}^{N-1}} \frac{g(y')}{|y'|^s(1+|y'|^{2-s})^{2(N-s)/(2-s)}}dy'\\
&\qquad\qquad\qquad \times \left(\int_{\mathbb{R}^{N-1}} \frac{|y'|^{2-2s}g(y')}{(1+|y'|^{2-s})^{2(N-s)/(2-s)}}dy'\right)^{-1}\\
&= (N-2)^{-2} \int_0^\infty \frac{r^{N-s}}{(1+r^{2-s})^{2(N-s)/(2-s)}}dr \times \left(\int_0^\infty \frac{r^{N+2-2s}}{(1+r^{2-s})^{2(N-s)/(2-s)}}dr\right)^{-1}
\end{align*}
Integration by parts gives for $2 \leq \beta \leq 2(N-s)-1$,
\[\int_0^\infty \frac{r^{\beta-2}}{(1+r^{2-s})^{\frac{2(N-s)}{2-s}-1}}dr = \frac{2N-2-s}{\beta-1}\int_0^\infty \frac{r^{\beta-s}}{(1+r^{2-s})^{\frac{2(N-s)}{2-s}}}dr.\]
Since
\[\int_0^\infty \frac{r^{\beta-s}}{(1+r^{2-s})^{\frac{2(N-s)}{2-s}}}dr = \int_0^\infty \frac{r^{\beta-2}}{(1+r^{2-s})^{\frac{2(N-s)}{2-s}-1}}dr-\int_0^\infty \frac{r^{\beta-2}}{(1+r^{2-s})^{\frac{2(N-s)}{2-s}}}dr,\]
we have
\begin{align}\label{min_id}
\int_0^\infty \frac{r^{\beta-s}}{(1+r^{2-s})^{\frac{2(N-s)}{2-s}}}dr = \frac{\beta-1}{2N-\beta-1-s}\int_0^\infty \frac{r^{\beta-2}}{(1+r^{2-s})^{\frac{2(N-s)}{2-s}}}dr.
\end{align}
So, plugging $\beta = N+2-s$ into \eqref{min_id}, we obtain
\[\lim_{\ep \rightarrow 0} \frac{I\!I(\ep)}{I(\ep)} = \frac{N-3}{(N+1-s)(N-2)^2}\]
and plugging $\beta = N+1-s$ into \eqref{min_id}
\begin{align*}
&\quad \frac{(N-s)}{(N-2)}\frac{K_1}{K_0}\\ &= \frac{N-s}{(N-2)^3}\left( \int_0^\infty \frac{r^{N-1-s}}{(1+r^{2-s})^{2(N-s)/(2-s)}}dr\right) \times \left(\int_0^\infty\frac{r^{N+1-2s}}{(1+r^{2-s})^{2(N-s)/(2-s)}}dr\right)^{-1}  \\&=(N-2)^{-2}.
\end{align*}
Therefore we obtain
\[\frac{I\!I(\ep)}{I(\ep)}< \frac{(N-s)}{(N-2)}\frac{K_1}{K_0}\]
for sufficiently small $\ep$ and complete the proof.


\section{Positivity of solution}
In this section, we establish the positivity of solutions. One first observes that 
\[0 = \langle J_\lambda'(u), u_-\rangle = \int_\Omega |\nabla u_-|^2 + \lambda|u_-|^2dx\]
where $u_- = \min(u,0)$. Since $\lambda > 0$, we have $u \geq 0$. Then the interior positivity of $u$ follows from the maximum principle
\begin{prop}
If $u \in C^1(\Omega \setminus \{x_1, x_2\})$ is a non-negative solution to \eqref{main in}, then 
\[u > 0 \text{ in }\Omega.\]
\end{prop}

\begin{proof}
We employ the argument in \cite{vaz}. If $u$ vanishes somewhere in $\Omega \setminus \{x_1, x_2\}$, then there exists $y_0 \in \Omega \setminus \{x_1, x_2\}$ and a ball $B=B_R(y_1)$ satisfying $u(y_0) = 0$, $\overline{B} \subset \Omega \setminus \{x_1, x_2\}, y_0 \in \partial B$ and $0 < u < a$ in $B$.
We observe that $u > 0$ on
\[A = \{ x :  \frac R2 < |x-y_1| < R\}\]
and
\[c = \inf\{u(x) : |x-y_1| = \frac R2\}\]
satisfies $0 < c <a$.

For given $k_1, k_2 > 0$, let $v(r)$ be solution to
\begin{align*}
\left\{\begin{array}{l}
v'' = k_1v' + k_2v \text{ for } 0 < r < \frac R2,\\
v(0) = 0, v(\frac R2) = c.
\end{array}
\right.
\end{align*}
We note that $v'(0) > 0$.
Now we consider
\[\overline u (x) = v(R - |x - y_1|).\]
Then $\overline u(x) = 0 \leq u(x)$ on $\partial B_R(y_1)$ and $\overline u(x) = c \leq u(x)$ on $\partial B_{\frac R2}(y_1)$.
Moreover, on $A$, we have
\begin{align*}
-\Delta \overline u + \lambda \overline u &= -v''(R - |x-y_1|) - v'(R - |x-y_1|)(\frac{1-N}{|x-y_1|}) + \lambda v(R - |x-y_1|)\\
&\leq (\frac{N-1}{R} - k_1)v'(R - |x-y_1|) + (\lambda - k_2)v(R - |x - y_1|)\\
&\leq 0
\end{align*}
for sufficiently large $k_1, k_2$.

We claim that $u \geq \overline u$ on $A$. Suppose not, there exists $\Omega_1 \subset A$ such that $\overline u > u$ on $\Omega_1$. And we have
\[-\Delta(\overline u - u) + \lambda (\overline u - u) \leq 0 \text{ on } \Omega_1.\]
So, by multiplying $\overline u - u$ and integrating over $\Omega_1$, we obtain
\begin{align*}
0 < \int_{\Omega_1} |\nabla(\overline u - u)|^2 + \lambda |\overline u - u|^2 dx \leq 0.
\end{align*}
which is a contradiction.

Since $u(y_0) = \overline u(y_0) = 0$, $u \geq \overline u$ on $A$ and $v' > 0$, $u'(y_1)$ should be positive which contradicts to $y_0$ is minimum point.
\end{proof}


\section{Regularity of solution to \eqref{main in}}
In this section, we verify the regularity of solution. Recall the following lemma. 
\begin{lem}[Appendix B in \cite{str}]
Let $N \geq 3$. Suppose $u \in H^1_{loc}(\Omega)$ is a weak solution to
\[-\Delta u = g(\cdot, u) \text { in } \Omega\]
where $g(x,u)$ is measurable in $x \in \Omega$ and continuous in $u \in \mathbb{R}$. If $g$ satisfies
\[g(x,u) \leq C(1+|u|^p)\]
for some $p \leq \frac{N+2}{N-2}$, then $u \in C^{1, \alpha}(\Omega)$ for any $\alpha > 0$.
\end{lem}

We observe that for $\Omega' \subset\subset \Omega \setminus{\{x_1, x_2\}}$, 
\[ \frac{|u|^{2^*(s)-2}u}{|x-x_1|^s} + \frac{|u|^{2^*(s)-2}u}{|x-x_2|^s} \leq C|u|^{2^*(s)-1}\]
and $2^*(s)-1 \leq \frac{N+2}{N-2}$. Hence the solution $u$ to \eqref{main in} is in $C^{1,\alpha}(\Omega')$.

\section{Neumann Problem with the multiple singularities}
In this section, we deal with the existence theory for the equation
\begin{align}\label{main in2}
\left\{
\begin{array}{l}
-\Delta u + \lambda u = \sum_{i=1}^I\frac{|u|^{2^*(s_i)-2}u}{|x-x_{i}|^{s_i}}\text{ in }\Omega\\
\frac{\partial u}{\partial \nu} = 0 \text{ on }\partial\Omega
\end{array}\right.
\end{align}
where $0 < s_i < 2$, $\Omega$ is $C^2$ bounded domain with $x_i \in \partial \Omega$ for $1 \leq i \leq I'-1$ and $x_i \in \Omega$ for $I' \leq i \leq I$. In addition, we assume that $x_{i_1} \neq x_{i_2}$ if $i_1 \neq i_2$. The energy functional is given by
\[J_\lambda(u) =  \frac12\int_{\Omega}(|\nabla u|^2 + \lambda u^2)dx - \sum_{i=1}^I\frac{1}{2^*(s_i)}\int_{\Omega}\frac{u_+^{2^*(s_i)}}{|x-x_i|^{s_i}}dx.\]
We see that $J_\lambda$ is $C^1$ and
\[\langle J'_\lambda(u), \phi \rangle = \int_\Omega( \nabla u \nabla \phi + \lambda u \phi) dx - \sum_{i=1}^I \int_\Omega \frac{u_+^{2^*(s_i)-1}}{|x-x_i|^{s_i}}\phi dx\]
for $\phi \in H^1(\Omega)$.

In the same fashion as the proof of Proposition \ref{ps}, we obtain the following proposition :
\begin{prop}\label{ps2}
The functional $J_\lambda$ satisfies the $(PS)_c$ condition for
\begin{align*}
\begin{array}{ll}
c < \min\left(\min_{1 \leq i \leq I'-1}\frac{2-s_i}{4(N-s_i)}S_{s_i}^{\frac{2^*(s_i)}{2^*(s_i)-2}}, \min_{I' \leq i \leq I}\frac{2-{s_i}}{2(N-{s_i})}S_{s_i}^{\frac{2^*(s_i)}{2^*(s_i)-2}}\right).
\end{array}
\end{align*}
\end{prop}

Using Proposition \ref{ps2}, one can obtain the following theorems by the same method as we prove for Theorem \ref{thm1}.
\begin{thm}[Existence of solution to \eqref{main in2} for small $\lambda$]\label{thm3}
There exists $\Lambda > 0$ such that \eqref{main in2} admit a positive solution for $\lambda$ with $0 < \lambda < \Lambda$.
\end{thm}

Moreover, under the geometric setting of $x_1 \in \partial \Omega$ and the mean curvature $H(x_1)$ is positive, one can prove the existence of  a positive solution to $\eqref{main in2}$ when \[\frac{2-s_1}{4(N-s_1)}S_{s_1}^{\frac{2^*(s_1)}{2^*(s_1)-2}} = \min\left(\min_{1 \leq i \leq I'-1}\frac{2-s_i}{4(N-s_i)}S_{s_i}^{\frac{2^*(s_i)}{2^*(s_i)-2}}, \min_{I' \leq i \leq I}\frac{2-{s_i}}{2(N-{s_i})}S_{s_i}^{\frac{2^*(s_i)}{2^*(s_i)-2}}\right).\] Actually we may assume $x_1 = (0, \cdots, 0) \in \partial \Omega$ and the mean curvature $H(0)$ is positive.
Then, up to rotation, the boundary near the origin can be represented by 
\[x_n = h(x') = \frac12 \sum_{i=1}^{N-1}\al_i x_i^2 + o(|x'|^2)\]
where $x' = (x_1, x_2, \cdots, x_{N-1}) \in D_\delta(0) = B_\delta(0) \cap \{x_N=0\}$ for some $\delta > 0$. Here $\al_1, \al_2, \cdots, \al_{N-1}$ are the principal curvature of $\partial \Omega$ at $0$ and the mean curvature $\sum_{i=1}^{N-1} \al_i > 0$. Denote
\[g(x') = \frac12\sum_{i=1}^{N-1}\al_ix_i^2.\] 

Consider
\begin{equation*}
U_\ep(x):=\ep^{\frac{N-2}{2(2-s)}}(\ep+|x|^{2-s})^{\frac{2-N}{2-s}}
\end{equation*}
for small parameter $\e>0$.
Then,
it follows that
\begin{equation*}
\int_{\re^N} |\nabla U_1|^2dx / (\int_{\re^N} \frac{|U_1|^{2^*(s)}}{|x|^s}dx)^{\frac{N-2}{N-s}} = S_s.
\end{equation*}
Choose $\delta$ such that $x_2, \cdots, x_I \notin B_{3\delta}(0)$. Set a cut-off function $\eta$ such that
\begin{equation*}
\quad \eta \in C_c^\infty(\re^N), \quad 0 \leq \eta \leq 1, \quad \eta =1 \ {\rm in} \ B_{\delta}(0), \quad\eta = 0 \ {\rm in}\ \re^N\setminus{B_{2\delta}(0)}.
\end{equation*}
Note that from $2^*(s) >2$, $J_\lambda(T\eta U_\ep) < 0$.
We define
\begin{align*}
\mathbb{P} =
\left\{p(t) \Big|
\begin{array}{ll}
 p(t) : [0,1] \rightarrow H^1(\Omega)\text{ is continous map with }\\
p(0) = 0 \in H^1(\Omega)\text{ and }p(1) = T\eta U_\ep(x)|_\Omega
\end{array}
\right\}.
\end{align*}
Let 
\[c^* = \inf_{p(t) \in \mathbb{P}}\sup_{0 \leq t \leq 1} \{ J_\lambda(p(t)) \}.\]
Then, thanks to Proposition \ref{ps2}, it suffices to show 
\begin{align*}
c^* < \frac{2-s_1}{4(N-s_1)} S_s^{\frac{2^*(s_1)}{2^*(s_1)-2}}.
\end{align*}
Denote
\begin{align*}
&\wt K_0^\ep := \int_{\Omega} |\nabla(\eta U_\ep)|^2 dx, \\
&\wt K_i^\ep := \int_\Omega \frac{(\eta U_\ep)^{2^*(s_i)}}{|x - x_i|^{s_i}}dx \text{ for } 1\leq i \leq I,\\
&\wt K_{I+1}^\ep := \int_\Omega (\eta U_\ep)^2dx.
\end{align*}
Then by repeating arguments in proof of Theorem \ref{thm2}, we get estimates for each $\wt K^\ep_i$ as follows :
\begin{itemize}
\item $\wt K_0^\ep$\begin{align*}
\wt K_0^\ep = 
\left\{\begin{array}{ll}
\frac12 \wt K_0 - C\ep^{\frac{1}{2-s}}|\ln\ep| + O(\ep^{\frac{1}{2-s}}) &\text{ when } N = 3,\\
\frac12 \wt K_0 - \wt I(\ep) + O(\ep^{\frac{1}{2-s}}) &\text{ when } N \geq 4,
\end{array}
\right.
\end{align*}
where
\begin{align*}
\wt I(\ep)&:= \int_{\mathbb{R}^{N-1}} \int_0^{g(x')} |\nabla U_\ep|^2 dx_Ndx',\\
\wt K_0 &:= \int_{\mathbb{R}^N} |\nabla U_\ep|^2dx.
\end{align*}
\item $\wt K_1^\ep$ \begin{align*}
\wt K_1^\ep = \frac12 \wt K_1 - \wt I\!I(\ep) + O(\ep^{\frac{1}{2-s_1}})
\end{align*}
where
\begin{align*}
\wt I\!I(\ep) &:= \int_{\mathbb{R}^{N-1}} \int_0^{g(x')} \frac{|U_\ep|^{2^*(s_1)}}{|x|^{s_1}}dx_Ndx',\\
\wt K_1 &= \int_{\mathbb{R}^N} \frac{|U_\ep|^{2^*(s_1)}}{|x|^{s_1}}dx.
\end{align*}
\item $\wt K^\ep_i$ for $2 \leq i \leq I$ \begin{align*}
\wt K^\ep_i= O(\ep^{\frac{s_i}{2-s_i}}) 
\end{align*}
\item $\wt K_{I+1}^\ep$ \begin{align*}
\wt K_{I+1}^\ep =
\left\{\begin{array}{ll}
O(\ep^{\frac{1}{2-s_1}}), & N=3,\\
O(|\ep^{\frac{2}{2-s_1}}\ln \ep|), & N=4,\\
O(\ep^{\frac{2}{2-s_1}}), & N \geq 5.
\end{array}
\right.
\end{align*}
\end{itemize}
Therefore using the fact 
\[\frac{\wt I\!I(\ep)}{\wt I(\ep)}< \frac{(N-s_1)}{(N-2)}\frac{\wt K_1}{\wt K_0}\]
which is verified in the proof of Theorem \ref{thm2}, we obtain the following theorem.

\begin{thm}[Existence of solution to \eqref{main in2} with the boundary singularity]\label{thm4}
Suppose $x_1 \in \partial \Omega$ and the mean curvature $H(x_1)$ is positive. If \[\frac{2-s_1}{4(N-s_1)}S_{s_1}^{\frac{2^*(s_1)}{2^*(s_1)-2}} = \min\left(\min_{1 \leq i \leq I'-1}\frac{2-s_i}{4(N-s_i)}S_{s_i}^{\frac{2^*(s_i)}{2^*(s_i)-2}}, \min_{I' \leq i \leq I}\frac{2-{s_i}}{2(N-{s_i})}S_{s_i}^{\frac{2^*(s_i)}{2^*(s_i)-2}}\right),\] then there exists a positive solution to $\eqref{main in2}$.
\end{thm}

\section*{Acknowledgments} 
The first author is supported by Grant-in-Aid for JSPS Research Fellow (JSPS KAKENHI Grant Number JP16J08945).


\end{document}